\documentclass[reqno]{amsart}
\usepackage{amssymb,amscd,verbatim,graphics,color,longtable}

\begin{document}

\setcounter{secnumdepth}{2}


\newcommand \und[2]{\underset{#1} {#2}}
\newcommand \fk[1]{{{\mathfrak #1}}}
\newcommand \C[1]{{\mathcal #1}}
\newcommand \ovl[1]{{\overline {#1}}}
\newcommand \ch[1]{{\check{#1}}}
\newcommand \bb[1]{{\mathbb #1}}
\newcommand \sL[1]{{^L {#1}}}
\newcommand \sT[1]{{^t {#1}}}
\newcommand \sV[1]{{^\vee{#1}}}

\newcommand \ti[1]{{\tilde #1}}
\newcommand \wti[1]{{\widetilde {#1}}}
\newcommand \un[1]{\underline{#1}}
\newcommand \unb[2]{\underset{#1}{{\underbrace{#2}}}}
\newcommand \ovb[2]{\overset{#1}{\overbrace{#2}}}

\newcommand\fg{\mathfrak g}
\newcommand\fh{\mathfrak h}
\newcommand\fa{\mathfrak a}
\newcommand\fz{\mathfrak z}
\newcommand\fm{\mathfrak m}
\newcommand\fn{\mathfrak n}
\newcommand\fp{\mathfrak p}

\newcommand \bA{{\mathbb A}}
\newcommand \bC{{\mathbb C}}
\newcommand \bF{{\mathbb F}}
\newcommand \bH{{\mathbb H}}
\newcommand \bR{{\mathbb R}}
\newcommand \bZ{{\mathbb Z}}

\newcommand\cha{{\check \alpha}}
\newcommand\chb{{\check \beta}}

\newcommand\one{1\!\!1}

\newcommand\CA{{\C A}}
\newcommand\CB{{\C B}}
\newcommand\CH{{\C H}}
\newcommand\CI{{\C I}}
\newcommand\CO{{\C O}}
\newcommand\CS{{\C S}}
\newcommand\CW{{\C W}}
\newcommand\CX{{\C X}}
\newcommand\CG{{\C G}}
\newcommand\CU{{\C U}}

\newcommand\CCI{{\C C(\CI)}}

\newcommand \LA{{\sL A}}
\newcommand\Lfa{{\sL {\fk a}}}
\newcommand \LG{{\sL G}}
\newcommand \LM{{\sL M}}
\newcommand \LB{{\sL B}}
\newcommand \Lg{{\sL{\fk g}}}

\newcommand\Ls{{\it C-symbol~ }}
\newcommand\DS{{\it Discrete series~ }}
\newcommand\Tr{{\it Tempered Representation~ }}
\newcommand\SpC{{\it Springer Correspondence~ }}
\newcommand\Sps{{\it S-symbol~ }}
\newcommand\LKT{{\it lowest K-type~ }}
\newcommand\LKTs{{\it lowest K-types~ }}
\newcommand\St{{\it Steinberg Representation~ }}
\newcommand\Sts{{\it Steinberg Representations~}}
\newcommand\sde[2]{{\ ^#1{#2}}}

\newcommand\ie{{\it i.e.}}
\newcommand\cf{{\it cf.~ }}
\newcommand\eg{{\it e.g.}}

\newcommand\ep{{\epsilon}}
\newcommand\la{{\lambda}}
\newcommand\om{{\omega}}
\newcommand\ome{{\omega}}
\newcommand\al{{\alpha}}
\newcommand\sig{{\sigma}}

\newcommand \vG{{\check G}}
\newcommand \vM{{\check M}}
\newcommand\oX{\ovl{X}}
\newcommand \vg{{\check \fk g}}
\newcommand \bXO{{\overline{X}(\C O)}}
\newcommand\bX{{\ovl{X}}}

\newtheorem*{theorem}{Theorem}
\newtheorem*{theorem 2}{Theorem 2}
\newtheorem*{corollary}{Corollary}
\newtheorem*{lemma}{Lemma}
\newtheorem*{proposition}{Proposition}
\newtheorem*{remark}{Remark}
\newtheorem*{definition}{Definition}
\newtheorem*{example}{Example}

\newcommand\Hom{\operatorname{Hom}}
\newcommand\Ind{\operatorname{Ind}}

\numberwithin{equation}{subsection}


\begin{title}{Reducibility of generic unipotent standard modules}\end{title}

\author{Dan Barbasch}
       \address[D. Barbasch]{Dept. of Mathematics\\
               Cornell University\\Ithaca, NY 14850}
       \email{barbasch@math.cornell.edu}

\author{Dan Ciubotaru}
        \address[D. Ciubotaru]{Dept. of Mathematics\\ University of
          Utah\\ SLC, UT 84112}
        \email{ciubo@math.utah.edu}


\thanks{This research was supported in part by the NSF grants DMS-0901104, 0968065, and 0967386.}

\begin{abstract}
Using Lusztig's geometric classification, we find the reducibility
points of a standard module for the affine Hecke algebra, in the case when
the inducing data is generic. This recovers the known result of
\cite{MS} for representations of split $p$-adic groups with
Iwahori-spherical Whittaker vectors. We also give a necessary
(but insufficient) condition for reducibility in the non-generic case. 
\end{abstract}

\maketitle




In \cite{L6}, the unipotent representations of a split $p$-adic group
$\C G$
of adjoint type are classified in terms of geometric data for the  dual
complex group $G$. More precisely, they are indexed by certain triples
$(\chi,\CO,\C L),$ where $\chi$ is a Weyl orbit of semisimple elements
in $G$, $\CO$ is a  ``graded'' orbit  in the Lie
algebra $\fg$, and $\C L$ is a local system on $\CO.$ This is realized
via equivalences with module categories for affine Hecke algebras
of {geometric type} constructed from $G$ (\cite{L2,L3}). 
It is shown in \cite{Re}, that in this correspondence, the unipotent
representations of $\C G$ admitting Whittaker vectors ({generic})
correspond to maximal orbits $\CO$ and trivial $\C L.$ For
Iwahori-spherical representations, the same result, with a different
proof, follows from \cite{BM} (and \cite{BM1}). 

In this paper, we determine explicitly, as a consequence of the
geometric classification, the reducibility points for the
standard representations (in the sense of Langlands classification)
when the inducing data is generic. This was known from \cite{CS} and \cite{MS}, as a
consequence of the Langlands-Shahidi method. In particular, our main
result, theorem \ref{sec:3.2} is essentially the same as proposition
3.3 in \cite{MS} (our parameter $\nu$ corresponds to the parameter $s$ in
there). We also show that for non-generic inducing data, the
reducibility points are necessarily a subset of those for the
corresponding generic case.  

For simplicity, we will work in the setting of the graded affine {Hecke
algebra} of \cite{L1}, and real central character (section
\ref{sec:1.2}), from which one can recover the representation theory
of the affine Hecke algebra (see section 4 in \cite{L6} for
example). Most of the paper is devoted to recalling the relevant
geometric results, particularly from \cite{L7}. Once they are in
place, the reducibility follows immediately by a simple comparison of
dimensions of  orbits. The essential result that we need is corollary \ref{c:lusztig}.

\smallskip

The information about reducibility of standard modules played an important
role in the determination of the generic Iwahori-spherical unitary dual
(equivalently, spherical unitary dual) of split $p$-adic groups of
exceptional types in \cite{BC}. In fact, this paper is mainly motivated by that work.

\section{Graded Hecke algebra}
\subsection{} Let $\fh$ be a finite dimensional vector space, 
$R\subset \fh^*$ a root system, with $\Pi=\{\al_1,\dots,\al_n\}$ the set
of simple roots, $\check R\subset \fh$ the set of coroots, and $W$ the Weyl group. Let $c:R\to\bZ_{>0}$ be a
function such that $c_\al=c_\beta$, whenever $\al$ and $\beta$ are $W$-conjugate.
As a vector space, \begin{equation}\label{1.1.1}\Bbb H=\Bbb
  C[W]\otimes\Bbb A,\end{equation} where $\Bbb A$ is the symmetric
algebra over $\fh^*$. The generators are $t_w\in \Bbb C[W]$, $w\in W$
and $\om\in\fh^*$. The relations between the generators are:
\begin{align}\label{1.1.2}
&t_wt_{w'}=t_{ww'},&\text{ for all }w,w'\in W;\notag\\
&t_s^2=1, &\text{ for any simple reflection } s\in W;\notag\\
&\om t_s=t_ss(\om)+c_\al\om(\check\alpha),&\text { for simple reflections } s=s_\alpha.\notag\\
\end{align}

\subsection{}\label{sec:1.2} By \cite{L1}, the center of $\bH$ is
  $\bA^W$.
On any simple (finite dimensional) $\bH$-module, the center of $\bH$ acts by
a character, which we will call a {\it central character}. The central
characters correspond to $W$-conjugacy classes of semisimple elements
$\chi\in \fh.$  We will assume throughout the paper that the
characters are {\it real}, \ie, hyperbolic. 

\subsection{}\label{sec:1.3} We present the {\it Langlands classification} for $\bH$ as
in \cite{E}. 
If $V$ is a (finite dimensional) simple $\bH$-module, $\bA$ induces a
generalized weight space decomposition 
$V=\bigoplus_{\lambda\in \fk h} V_\lambda.$
 Call $\lambda$ a {\it weight of $V$} if $V_\lambda\neq 0.$

\begin{definition} The irreducible module $\sigma$ is called {\it
     tempered} if $\ome_i(\lambda)\le 0,$ for every
  weight $\lambda\in \fk h$ of $\sigma$ and every fundamental
  weight $\ome_i\in\fk h^*,$ and in addition, $\lambda$ is zero on the
  real span of the set $x\in \fh^*$ perpendicular on coroots.  
If $\sigma$ is tempered, and $\ome_i(\lambda)<0,$ for all
  $\lambda,\omega_i$ as above, $\sigma$ is called a {\it discrete series}.
\end{definition}

For every
$\Pi_P\subset \Pi$, define $R_M\subset R$ to be the set of roots generated by
$\Pi_P$,  $\check{R}_M\subset \check R$ the corresponding set of
coroots, and $W_P\subset W$ the corresponding Weyl subgroup.  (The
notation will make more sense in the sequel, when $P=MN$ will denote a
parabolic subgroup of the complex reductive group $G$.)

Let
$\bH_M$ be the Hecke algebra attached to $(\fh, R_M)$. It can be regarded
naturally as a subalgebra of $\bH.$ 

Define $\fk t=\{\nu\in\fk h:
\langle\al,\nu\rangle=0,\text{ for all } \al\in\Pi_P\}$ and $\fk
t^*=\{\lambda\in\fk h^*: \langle\lambda,\check\al\rangle=0,\text{ for
  all }\al\in \Pi_P\}.$ 
Then  $\bH_M$ decomposes as
$$\bH_M=\bH_{M_s}\otimes S(\fk t^*),$$ where $\bH_{M_s}$ is the Hecke
algebra attached to $(\bC\langle\Pi_P\rangle,R_M).$ 

We will denote by $I(P,U)$ the induced module
$I(P,U)=\bH\otimes_{\bH_M} U.$ 

\begin{theorem}[\cite{E}]\label{t:1.4} 
\begin{enumerate}
\item
Every irreducible $\bH$-module
  is a quotient  of a standard induced module
  $X(P,\sigma,\nu)=I(P,\sigma\otimes \bC_\nu),$ where $\sigma$ is
  a tempered module for $\bH_{M_s},$ and $\nu\in \fk t^+=\{\nu\in\fk t:
    \al(\nu)>0,\text{ for all }\al\in\Pi\setminus\Pi_P\}.$ 

\item Assume the notation from (1). Then $X(P,\sigma,\nu)$ has a unique
  irreducible quotient, denoted by $L(P,\sigma,\nu)$.

\item If $L(P,\sigma,\nu)\cong L(P',\sigma',\nu'),$ then $\Pi_P=\Pi_{P'},$ $\sigma\cong \sigma'$ as
$\bH_{M_s}$-modules, and $\nu=\nu'.$
\end{enumerate}
\end{theorem}

We will call a triple $(P,\sigma,\nu)$ as in theorem \ref{t:1.4}, a {\it
  Langlands parameter}.

\section{Geometric parameterization}\label{sec:2}

\noindent{\bf Notation.} In the following, whenever $Q$ denotes a
complex Lie group, $Q^0$ will be the identity component, and $\fk q$ will
denote the Lie algebra. If $s$ is an element of $Q$ or $\fk q$, we
will denote by $Z_Q(s)$ the centralizer in $Q$ of $s.$

\subsection{}\label{sec:2.1} Let $G$ be a reductive connected complex algebraic group, with Lie
algebra $\fg.$ Let $B$ be a Borel subgroup, and $A\subset B$ a maximal
torus, and denote by $\Delta$ the roots of $A$ in $G$, and by $\Delta^+$, the roots of $A$ in $B$.

Let $S=LU$ denote a parabolic subgroup, with $\fk s=\fk
l+\fk u$ the corresponding Lie algebras, such that $\fk l$ admits an
irreducible $L$-equivariant
cuspidal local system  $\Xi$ on a
nilpotent $L$-orbit $\C C\subset \fk l$ (as in \cite{L2},\cite{L5}). The classification of
cuspidal local systems can be found in \cite{L5}. In particular, $W=N(L)/L$ is
a Coxeter group. 

Let $H$ be the center of $L$ with Lie algebra $\fk h$, and let $R$ be
the set of nonzero weights $\al$ for the $ad$-action of $\fk h$ on $\fk g,$
and $R^+\subset R$ the set of weights for which the corresponding
weight space $\fg_\al\subset\fk u.$ For each parabolic $S_j=L_jU_j$, $j=1,n$, such that
$S\subset S_j$ maximally and $L\subset L_j$, let $R_j^+=\{\al\in R^+:
\al(\fz(\fk l_j))=0\},$ where $\fz(\fk l_j)$ denotes the center of
$\fk l_j.$ It is shown in \cite{L2} that each $R_j^+$ contains a unique
$\al_j$ such that $\al_j\notin 2R.$ 

Let $Z_G(\C C)$ denote the centralizer in $G$ of a Lie triple for $\C
C,$ and $\fz(\C C)$ its Lie algebra. 

\begin{proposition}[\cite{L2}]
\begin{enumerate}
\item $R$ is a (possibly non-reduced) root system in $\fk h^*$, with
simple roots $\Pi=\{\al_1,\dots,\al_n\},$ with Weyl group $W.$

\item $H$ is a maximal torus in $Z^0=Z^0_G(\C C)$.

\item $W$ is isomorphic to $W(Z^0_G(\C C))=N_{Z^0}(H)/H.$

\item The set of roots in $\fz(\C C)$ with respect to $\fk h$ is exactly
the set of non-multipliable roots in $R.$  
\end{enumerate}
\end{proposition}

For each $j=1,\dots,n$, let $d_j\ge 2$ be such that 
\begin{equation}
(ad(e)^{d_j-2}:\fk l_j\cap\fk u\to\fk l_j\cap\fk u) \neq 0, \text{ and }
(ad(e)^{d_j-1}:\fk l_j\cap\fk u\to\fk l_j\cap\fk u) =0.
\end{equation}

By proposition 2.12 in \cite{L2}, $d_{i}=d_j$ whenever $\al_i$ and
$\al_j$ are $W$-conjugate. Therefore, as in
(\ref{1.1.1}),(\ref{1.1.2}), we can define a Hecke algebra
$\bH_S$ with parameters $c_j=d_j/2$.   The
explicit algebras which may appear are listed in 2.13 of
\cite{L2}. The case of Hecke algebras with equal
parameters $c_j=1$, arises when one takes $S=B$, $R=\Delta,$   and $\C C$
and $\Xi$ to be trivial. 

If $P\subset G$ is a parabolic subgroup, such that $S\subset P,$ then
denote 
\begin{equation}\label{eq:2.1.2}
\Pi_{P/S}=\{\al_j\in\Pi: S_j\subset P\}.
\end{equation}
When $S=B,$ we write just $\Pi_P.$

\smallskip

Let us denote by $\Phi(G)$ the set of graded Hecke algebras $\Bbb H_S$ obtained
by the above construction. The unique Hecke algebra with equal
parameters in $\Phi(G)$ will be denoted by $\bH_0.$

\subsection{}\label{sec:2.2}

 Fix a (hyperbolic) semisimple element ({\it an infinitesimal character}) $\chi\in \fa,$  and set 
\begin{equation}
G_0=\{g\in G: Ad(g)\chi=\chi\},\quad
\fg_t=\{y\in\fg:[\chi,y]=ty\},\ t\in\bR. 
\end{equation}

Note that 
\begin{equation}\label{eq:roots}
\fg_t=\left\{\begin{matrix}\displaystyle{\bigoplus_{\al\in\Delta,\al(\chi)=t}\fg_\al},
&t\neq 0\\\displaystyle{\fa\oplus  \bigoplus_{\al\in\Delta,\al(\chi)=0}\fg_\al},
&t= 0 \end{matrix}\right..
\end{equation}

For $\bH\in\Phi(G)$, corresponding to $S=LU$ in the notation of section
\ref{sec:2.1}, denote by $\text{mod}_\chi\bH$ the category of
finite dimensional $\bH$-modules of central character equal to
the projection of $\chi$ onto $\fh.$

\begin{theorem}[\cite{L3}]\label{t:2.2} There exists a one-to-one correspondence
  between the standard (or irreducible) objects in
  $\sqcup_{\Bbb H\in \Phi(G)}\text{mod}_\chi(\Bbb H)$ and the set of
  pairs $\xi=(\CO,\C L)$, where 
\begin{enumerate}
\item $\CO$ is a $G_0$-orbit on $\fg_1.$
\item $\C L$ is an irreducible $G_0$-equivariant local system on
  $\CO$. 
\end{enumerate}

\end{theorem}

We say that two modules in $\sqcup_{\Bbb H\in
  \Phi(G)}\text{mod}_\chi(\Bbb H)$ are in the same {\it L-packet} if
they correspond to the same orbit $\CO.$

For  $\Bbb H_0$-modules, the local systems which appear are of {\it Springer
  type} (\cite{L7}). More precisely, if $e\in \CO$, then  $\C L$
  corresponds to a representation $\phi$ of the component group
  $Z_{G_0}(e)/Z_{G_0}(e)^0$. The representations $\phi$ which are
  allowed must be in the
  restriction $Z_{G_0}(e)/Z_{G_0}(e)^0\subset Z_G(e)/Z_G(e)^0$ of
  a representation which appears in Springer's correspondence. In particular, the trivial local systems always parameterize $\bH_0$-modules.

\subsection{}\label{sec:2.3}
Let $Orb_1(\chi)$ denote the set of $G_0$ orbits on $\fg_1.$

\begin{enumerate}
\item $Orb_1(\chi)$ is finite.
\item For every $\CO\in Orb_1(\chi),$ $\overline\CO\setminus \CO$ is
  the union of some orbits $\CO'$ with $\dim\CO'<\dim\CO.$
\item There is a unique open (dense) orbit $\CO_{open}$ in $Orb_1(\chi).$ 
\end{enumerate}

In other words, $\fg_1$ is a prehomogeneous vector space with finitely
many $G_0$-orbits.  
A parameterization of $Orb_1(\chi)$ appeared in \cite{Ka}. We will
instead use, in sections \ref{sec:2.4} and \ref{sec:2.5}, the
formulation of \cite{L7}. 

\subsection{}\label{sec:2.2a} By \cite{L4}, the categories $\text{mod}_\chi\Bbb H,$
$\Bbb H\in \Phi(G)$, have tempered modules if and only if $\chi$
is the middle element of a nilpotent orbit in $\fg.$ In this case the
standard modules parameterized by $(\CO_{open},\C L)$ are irreducible
and they exhaust the tempered modules.
If in addition, $\chi$ is the middle element of a distinguished
nilpotent orbit, then the tempered modules are discrete series.

\subsection{}\label{sec:2.3a} By \cite{Re}, there is a unique {\it generic}
module in $\sqcup_{\Bbb H\in \Phi(G)}\text{mod}_\chi(\Bbb H)$, which is
parametrized by $(\CO_{open},triv),$ where $triv$ denotes the trivial local
system. Note that this is always a module of $\Bbb H_0.$ The fact that
the generic module in $\text{mod}_\chi(\Bbb H_0)$ is parameterized by
$(\CO_{open},triv)$ is also an immediate consequence of the results in
\cite{BM} and \cite{BM1}. In \cite{BM}, it is proven that the generic
$\bH_0$-module is characterized by the property that it  contains the
$sign$ representation of $W.$ 

\subsection{}\label{sec:2.4} Let $e$ be a representative of an orbit
$\CO=\CO_e$ in $\fg_1.$ To $e$, one associates, conform \cite{L7}, a parabolic subalgebras of
$\fg$, which will be denoted by $\fp^e$. It will be used to give a
parameterization of $Orb_1(\chi).$ 

By the graded version of the Jacobson-Morozov triple (\cite{L7}),
$e\in \fg_1$ can be
embedded into a Lie
triple $\{e,h,f\}$, such that $h\in \fa\subset\fg_0,$ and $f\in
\fg_{-1}.$ Define a gradation of $\fg$ with respect to $\frac 12h$ as well,
\begin{equation}\fg^r=\{y\in\fg: [\frac 12 h,y]=ry\},\ r\in \frac 12\bZ,
\end{equation}
 and set 
\begin{equation}
\fg^r_t=\fg_t\cap\fg^r.
\end{equation}
 Then
\begin{equation}\fg=\bigoplus_{t,r}\fg^r_t.
\end{equation}
Set
\begin{equation}\label{par1}\fm^e=\bigoplus_{t=r}\fg_t^r,\quad
  \fn^e=\bigoplus_{t<r}\fg_t^r,\quad
  \fp^e=\fm^e\oplus\fn^e.\end{equation} 
Clearly, $\fa\subset\fg_0^0\subset\fm^e.$

\begin{definition}
One says that $\chi$ is {\it rigid} for a Levi subalgebra $\fm,$ if
$\chi$ is congruent modulo $\fz(\fm)$ to a middle element of a
nilpotent orbit in $\fm$. 
\end{definition}

\noindent Whenever $Q$ is a subgroup with Lie algebra $\fk q,$ we will write
$Q_0=Q\cap G_0$ and $\fk q_t=\fk q\cap \fg_t.$

We record the important properties of $\fp^e.$

\begin{proposition}[\cite{L7}]\label{p:2.4} Consider the subalgebra $\fp^e$ defined by
  (\ref{par1}), and let $P^e$ be the corresponding parabolic subgroup.

\begin{enumerate}

\item $\fp^e$ depends only on $e$ and not on the entire Lie triple
  $\{e,h,f\}.$

\item $\chi$ is {rigid} for $\fm^e.$ 

\item $e$ is an element of the
  open $M_0^e$-orbit in $\fm^e_1.$ 

\item The $P^e_{0}$-orbit of $e$ in $\fp^e_{1}$ is open, dense in $\fp^e.$ 

\item  $Z_{G_0}(e)\subset P^e$.

\item The inclusion
  $Z_{M_0^e}(e)\subset Z_{G_0}(e)$ induces an isomorphism
  of the component groups.
\end{enumerate}

\end{proposition}

An immediate corollary of (4) and (5) in the proposition is a dimension
formula for the orbits in $Orb_1(\chi).$

\begin{corollary}[Lusztig]\label{c:lusztig} For an orbit $\CO_e\in
  Orb_1(\chi),$ \begin{equation}\dim \CO_e=\dim
    \fp_1^e-\dim\fp_0^e+\dim\fg_0,\end{equation}
where $\fp^e_i=\fp^e\cap\fg_i,$ $i=0,1.$
\end{corollary}

\subsection{}\label{sec:2.5}
{\bf Definition.} {\it A parabolic subgroup $P$ with Lie algebra $\fp$ is
  called {\it good} for $\chi$ if $\fp=\fp^e$ for some
  nilpotent $e\in \fg_1$ (notation as in (\ref{par1})), and such that
  it satisfies (2) in  proposition \ref{p:2.4}.} 

\medskip

Let $\C P(\chi)$ denote the set of good parabolic
subgroups for $\chi.$ The parameterization of $Orb_1(\chi)$ is as follows.

\begin{theorem}[\cite{L7}]\label{t:2.5}
The map $\CO_e\mapsto P^e$ defined in section \ref{sec:2.4}
induces a bijection between $Orb_1(\chi)$ and $G_0$-conjugacy
classes in $\C P(\chi)$.
\end{theorem} 

\begin{proof}
The definition of the inverse map is at follows. Let $P=MN$ be a good parabolic for
$\chi.$ Then there exists $s$ a middle element of a Lie triple in
$\fm,$ such that  $\chi\equiv s$ (mod
$\fz(\fm)$). Moreover, the decomposition (\ref{par1}) must hold with
respect to $\chi$ and $s.$ Let $G_0^0\subset G_0$ be the reductive
subgroup whose Lie algebra is $\fg_0^0$ (notation as in section
\ref{sec:2.4}).  Then $G_0^0$ acts on
$\fg_1^1$, and there is a unique open orbit of this action. Let
$\CO$ be the unique $G_0$-orbit on $\fg_1$ containing it. The
inverse map associates $\CO$ to $P.$  
\end{proof}

\section{Reducibility points}\label{sec:3}

\subsection{}\label{sec:3.1} Let $\{e,h,f\}$ be a graded Lie triple
for the orbit $\CO_e\in Orb_1(\chi).$ Assume that $\fp=\fm+\fn$ is a
standard parabolic subalgebra, $\fk b\subset\fp$, such  that $\{e,h,f\}\subset\fm.$ Let
$\bar\fp=\fm+\bar\fn$ be the opposite parabolics subalgebra. Let
$\Pi_P\subset\Pi$ denote the simple roots defining $P,$ and denote by
$\Delta_M$ and $\Delta_N$ the roots in $\fm,$ respectively $\fn.$ We can  write
$$\chi=\frac 12 h+\underline\nu, \text{ with }\underline\nu\in\fz_G(e,h,f).$$  
 
\begin{lemma}
Let $\{e,h,f\},\chi$ be as before, and assume that $\chi=\frac 12
h+\underline\nu$ has $\underline\nu$ dominant with respect to $\Delta_N.$ Then:

\begin{enumerate}
\item $\fm^e=\fm=\fz_\fg(\underline\nu).$
\item $\fp^e=\bar\fp.$
\end{enumerate}

In particular, $\bar\fp$ is a good
parabolic for $\chi.$ 
\end{lemma}

\begin{proof}
The first assertion is obvious by the definitions. From (\ref{par1})
and the dominance conditions, we also see immediately
that $\bar\fn=\fn^e.$ 
\end{proof}

Let $\sigma$ be the tempered module of $\Bbb H_{M_s}$ (notation as in
\ref{sec:1.3}) parameterized by $\{e,h,f\}.$ By the classification theorems
of \cite{L3} and \cite{L4}, we known that, in the correspondences of theorem \ref{t:2.2}, the standard module
$X(P,\sigma,\nu)$, and the Langlands quotient $L(P,\sigma,\nu)$  are
parameterized in $Orb_1(\chi)$ by the orbit $G_0\cdot e.$ Therefore,
in theorem \ref{t:2.5}, they correspond to the parabolic subalgebra
$\bar\fp$.  

\subsection{}\label{sec:3.2} Now assume that $\fp=\fm+\fn$ is a {\it maximal}
parabolic of $\fg.$ Then $\Pi\setminus\Pi_P=\{\al\}$. Let
$\check\ome$ denote the fundamental coweight for $\al.$

As before, let $\sigma$ be the tempered module attached to the map
\begin{equation}\label{eq:param}\fk {sl}(2)=\bC\langle
  e,h,f\rangle\hookrightarrow \fk m.
\end{equation} 
Then $\fk n$ is an $\fk{sl}(2)$-module, via the adjoint action of $\fk
m.$ Let $k(\al)$ denote the multiplicit with which $\al$ appears in
the highest root for $\Delta.$ \footnote{If  $\fg$ is a classical
  simple algebra, this multiplicity is always $1$ or $2$.}

The coweight $\check\ome$ commutes with the $\fk{sl}(2)$. Decompose
$\fk n$ as $\fk n=\oplus_{i=1}^{k(\al)}\fk n_i,$ where
$\fk n_i$ is the $i$-eigenspace of $\check \omega.$ 
Then
decompose each $\fk n_i$ into simple $\fk{sl}(2)$-modules \begin{equation}\label{eq:dec}\fk
n_i=\oplus_j (d_{ij}),\ i=1,\dots,k(\al),\end{equation} where $(d)$ is the simple $\fk
{sl}(2)$-module of dimension $d$.

\begin{theorem}\label{main}Let $\fp=\fm+\fn$ be a maximal parabolic, and $\sigma$ be a generic tempered
  module parameterized by (\ref{eq:param}). Then the reducibility
  points $\nu>0$ of the standard $\bH_0$-module $X(P,\sigma,\nu)$
  are 
\begin{equation}
\nu\in \left\{\frac{d_{ij}+1}{2i}\right\}_{i,j},
\end{equation}
where the integers $d_{ij}$ are defined in (\ref{eq:dec}).
Equivalently, these are the zeros of the rational function in $\nu$,
\begin{equation}\label{eq:red}
\prod_{\beta\in \Delta_N}
\frac{1-\langle\beta,\chi\rangle}{\langle\beta,\chi\rangle},
\end{equation}
where $\chi=\frac 12h+\nu\check\ome$ is the infinitesimal character of
$X(P,\sigma,\nu).$  
\end{theorem}

\begin{proof} 
Let $\CO(\bar\fp)$ be the orbit parameterizing $X(P,\sigma,\nu)$ by
section \ref{sec:3.1}. By section \ref{sec:2.3a}, $X(P,\sigma,\nu)$ is
irreducible if and only if $\CO(\bar\fp)=\CO_{open}.$

Corollary \ref{p:2.4} implies that 
$\dim\CO(\bar\fp)=\dim \fg_0-\dim (\fg_0\cap\bar\fp)+\dim (\fg_{1}\cap\bar\fp).$
From this and the fact that $\dim \CO_{open}=\dim \fg_{1},$
it follows, by equation (\ref{eq:roots}), that $\CO(\bar\fp)=\CO_{open}$ if and only if
\begin{equation}
\#\{\beta\in \Delta_N: \langle\beta,\chi\rangle=1\} = \#\{\beta\in \Delta_N: \langle\beta,\chi\rangle=0\}.
\end{equation}
Consider the rational function of $\nu$, $\prod_{\beta\in \Delta_N}
\frac{1-\langle\beta,\chi\rangle}{\langle\beta,\chi\rangle}.$  Therefore, the
reducibility points are given by the zeros of this function.

\smallskip

The explicit list of reducibility points follows from the fact that
$\{\langle\beta,h\rangle: \beta\in
\Delta_N\}=\sqcup_{i,j}\{d_{ij}-1,d_{ij}-3,\dots,-d_{ij}+1\},$ and so 
\begin{equation}
\prod_{\beta\in \Delta_N}\frac {1-\langle
  \chi,\beta\rangle}{\langle
  \chi,\beta\rangle}=\prod_{i,j}\frac {\frac
  {d_{ij}+1}{2i}-\nu}{\frac {d_{ij}-1}{2i}+\nu}.
\end{equation}
\end{proof}

We remark that in the proof of formula (\ref{eq:red}), one does not
use the assumption that $\fp$ be maximal parabolic. This formula holds
as is for any parabolic $\fp.$ 

\smallskip

\noindent{\bf Example.} The most interesting example of reducibility
points for maximal parabolic induction is the case $\Pi_P=A_4+A_2+A_1$
in $\Pi=E_8,$ with $e$ the principal nilpotent in $A_4+A_2+A_1$ (which
means that $\sigma$ is the Steinberg representation). Then $k(\al)=6,$
$\dim\fn=106,$ and 
the $sl(2)$ decompositions (\ref{eq:dec}) are
\begin{align}\notag
&\fn_1=(8)+2\cdot(6)+2\cdot (4)+(2) \qquad
  &\fn_2=(9)+(7)+2\cdot(5)+(3)+(1)\\\notag
&\fn_3=(8)+(6)+(4)+(2)\qquad
&\fn_4=(7)+(5)+(3) \\
&\fn_5=(4)+(2) \qquad&\fn_6=(5).
\end{align}
There are $11$ reducibility points: 
\begin{equation}\left\{\frac 3{10},\frac 12,\frac 34,\frac
56,1,\frac 76,\frac 32,2,\frac 52,\frac 72,\frac
92\right\}.\end{equation} 

\subsection{} One also immediately obtains a partial result for
  non-generic data. Recall the notation and construction of section
  \ref{sec:2.1}. In particular, if $\sigma'$ is parameterized by (\ref{eq:param}), there exists a unique triple
  $(S,\C C,\Xi)$ such that $\sigma'$ is a discrete series for the
  subalgebra $\bH_{S,\Pi_{P/S}}$ in $\bH_S.$ 

\begin{proposition}
Let $\sigma$ and $\sigma'$ be tempered modules in the L-packet parameterized
by (\ref{eq:param}), and assume that $\sigma$ is generic. The standard
$\bH_S$-module $X(P/S,\sigma',\nu)$ is reducible for $\nu>0$ {\it only if} the
standard $\bH_0$-module $X(P,\sigma,\nu)$ is reducible. 
\end{proposition}

\begin{proof}
If $X(P/S,\sigma',\nu)$ is reducible, then the corresponding orbit is not
the open orbit. But this means $X(P,\sigma,\nu)$ is reducible as well. 
\end{proof}

\noindent{\bf Remark.} This result gives necessary conditions for reducibility, but not
sufficient. In fact, they are far from being sharp for non-generic
inducing data as seen in the following example.

\smallskip

\noindent{\bf Example.} Consider $\bH_0$ of type $C_{n+1},$ and $\fp$ of
type $C_n,$ and assume that $n$ is a triangular number. Let the nilpotent
element $e$ correspond to the distinguished orbit $(2,4,\dots,2k)$ in
$\fk{sp(2n)},$ and $\chi$ be half the middle element of a Lie triple for $e.$ 

There are $\left(\begin{matrix} k\\ \lfloor\frac k 2\rfloor\end{matrix}\right)$ discrete
series in $\text{mod}_\chi\bH_0(C_n).$ Let $\sigma$ be the generic one. There exists
a unique nongeneric discrete series, call it $\sigma'$, characterized by the fact that $\sigma'|_{W(C_n)}$ is irreducible. More precisely, $\sigma'|_{W(C_n)}=\mu_k$, where 
\begin{equation}
\mu_k=\left\{\begin{matrix} m^{2m+1}\times 0,&\text{if }k=2m\\
    0\times (m+1)^{2m+1}, &\text{if }k=2m+1.\end{matrix}\right.
\end{equation}
 (The notation for $W(C_n)$-representations, and the algorithms for the Springer correspondence are in \cite{L5}.)

Theorem \ref{main} implies that the reducibility points, $\nu>0$,
for $X(C_n,\sigma,\nu)$ are $$\nu\in\left\{\frac 12,\ \frac 32,\ \frac
52,\dotsc,\ k+\frac 12\right\},$$ 
but one can show, using the $W(C_{n+1})$-structure, that the only reducibility points of $X(C_n,\sigma',\nu)$ are 
$$\nu\in\left\{\lfloor\frac k 2\rfloor+\frac 12, \ k+\frac 12\right\}.$$

\end{document}